\documentclass[11pt,reqno]{amsart} 
\date{}
\usepackage{cite} 
\usepackage{amsmath,amsthm,amstext,amscd}
\usepackage{amsfonts}
\usepackage{enumerate}
\usepackage{amssymb}

\theoremstyle{plain}
      \newtheorem{theorem}{Theorem}[section]

      \theoremstyle{definition}

      \theoremstyle{remark}

\numberwithin{equation}{section}  

\title[On  a representation of Humbert's double hypergeometric series]{\bf  ON A REPRESENTATION OF HUMBERT'S DOUBLE HYPERGEOMETRIC SERIES $\Phi_2$ IN A SERIES OF GAUSS'S  $_2F_1$ FUNCTION} 
\author[Arjun K. Rathie]{\bf  Arjun K. Rathie } 

\address{ Arjun K. Rathie, Department of Mathematics, School of Mathematical and Physical Sciences, Central University of Kerala, Riverside Transit Campus,
Padennakkad P.O. Nileshwar, Kasaragod- 671 328, Kerala - INDIA}
\email{akrathie@gmail.com}

\begin{document}
\maketitle

\begin{abstract}
In this paper, we aim to obtain a  representation of Humbert's hypergeometric function $\phi_2$  in a series of Gauss's function $_{2}F_1$. A few interesting results have also been deduced as  special case of our main findings.
\end{abstract}



\section{Introduction}

We start with the one of the Humbert's function $\phi_2$ defined by [2,3]
\begin{equation} \label{phi2}
\phi_2(a, b; c; x,y) = \sum_{n=0}^\infty \sum_{k=0}^\infty \frac{(a)_n (b)_k}{(c)_{n+k}} . \frac{x^n y^k}{n! \: k!}
\end{equation}
The series \eqref{phi2} converges absolutely for all $x, y \in \mathbb{C}$.\\

In the theory of hypergeometric series, classical summation theorems such as those of Gauss's summation theorem[4]
\begin{equation} \label{gauss-summation}
_{2}F_1\left[\begin{array}{c}a, b\\ c \end{array} ; 1\right] = \frac{\Gamma(c) \Gamma(c-a-b)}{\Gamma(c-a) \Gamma(c-b)}
\end{equation}
provided $\Re(c-a-b) > 0$.\\
and Kummer's summation theorem[4]
\begin{equation} \label{kummer-summation}
_{2}F_1\left[\begin{array}{c}a, b\\ 1+a-b \end{array} ; -1\right] = \frac{\Gamma(1+\frac{1}{2}a) \Gamma(1+a-b)}{\Gamma(1+a) \Gamma(1+\frac{1}{2}a-b)}
\end{equation}
play an important role.\\

The aim of this short research note is to find a  representation of Humbert's hypergeometric function $\phi_2$ in a series of Gauss's function $_{2}F_1$. In the end,  a few   interesting results have also been obtained as   special case of our main findings.
%===========================
\section{MAIN RESULT}
%===========================
The result to be proved in this short research note is given in the following theorem.\\
\begin{theorem}
For $c \neq 0, -1, -2, \cdots$, the following result holds true.
\begin{equation} \label{main}
\phi_2(a, b; c; x,y)  = \sum_{m=0}^\infty \frac{(a)_m}{(c)_m} \; _{2}F_1\left[\begin{array}{c}-m, b\\ 1-a-m \end{array} ; \frac{y}{x}\right] \frac{x^m}{m!}
\end{equation}
\end{theorem}
%===============================================================
\begin{proof}
In order to prove the result \eqref{main}, we proceed as follows. Denoting the left-hand side of \eqref{main} by S, we have upon using  \eqref{phi2}
\begin{equation*}
S =   \sum_{m=0}^\infty \sum_{n=0}^\infty \frac{(a)_{m} (b)_n}{(c)_{m+n}} . \frac{x^{m} y^n}{(m)! \: n!}
\end{equation*}   
Now replacing $m$ by $m-n$ and using a known result[4, lemma 10, p.56, eqn.(1)]
\begin{equation*}
  \sum_{n=0}^\infty \sum_{k=0}^\infty A(k,n) =    \sum_{n=0}^\infty \sum_{k=0}^n A(k,n-k)
\end{equation*} 
we have
\begin{equation*}
S =   \sum_{m=0}^\infty \sum_{n=0}^m \frac{(a)_{m-n} (b)_n}{(c)_{m}} . \frac{x^{m-n} y^n}{(m-n)! \: n!}
\end{equation*}   
Using the elementary identites
\begin{equation*}
(a)_{m-n} = \frac{(-1)^n (a)_m}{(1-a-m)_n}
\end{equation*}
and 
\begin{equation*}
(m-n) ! = \frac{(-1)^n m!}{(-m)_n}
\end{equation*}
we have, after some simplification
\begin{equation*}
S =   \sum_{m=0}^\infty  \frac{(a)_{m}}{(c)_{m}} \frac{x^{m}}{m!} . \sum_{n=0}^m  \frac{(-m)_{n}(b)_{n}}{(1-a-m)_n } \left(\frac{y}{x}\right)^n
\end{equation*}   
summing up the inner series, we easily arrive at the right-hand side of  \eqref{main}. This completes the proof of the theorem.
\end{proof}
%================================
\section{SPECIAL CASES}
In this section we shall mention a few interesting  special cases of our theorem.\\
%=========
\textbf{1.}  In \eqref{main}, if we take $y=x$, we have 
\begin{equation} \label{mainsc1}
\phi_2(a, b; c; x,x)  = \sum_{m=0}^\infty \frac{(a)_m}{(c)_m} \; _{2}F_1\left[\begin{array}{c}-m, b\\ 1-a-m \end{array} ; 1\right] \frac{x^m}{m!}
\end{equation}
We observe here that the $_{2}F_1$ appearing on the right-hand side can be evaluated with the help of classical Gauss's summation theorem \eqref{gauss-summation} and after some simplification, we get 
\begin{equation*} \label{mainsc1a}
\phi_2(a, b; c; x,x)  =  _{1}F_1\left[\begin{array}{c} a+b\\ c \end{array} ; x\right] 
\end{equation*}
a result recently obtained by Manako[2, p.506, eqn.(10)] and also recorded in [5, p.322, eqn(187)].\\
\textbf{ 2.} In \eqref{main} if we take $b=a$ and $y=-x$, we get
\begin{equation} \label{mainsc2}
\phi_2(a, a; c; x, -x)  = \sum_{m=0}^\infty \frac{(a)_m}{(c)_m} \; _{2}F_1\left[\begin{array}{c}-m, a\\ 1-a-m \end{array} ; -1\right] \frac{x^m}{m!}
\end{equation}
Again, we observe that the $_{2}F_1$ appearing on the right-hand side can be evaluated with the help of classical Kummer's summation theorem \eqref{kummer-summation} and after some simplification, we get
\begin{equation} \label{mainsc2a}
\phi_2(a, a; c; x, -x)  = \;  _{1}F_2\left[\begin{array}{c} a\\ \frac{1}{2}c , \frac{1}{2}c+\frac{1}{2} \end{array} ; \frac{x^2}{4}\right] 
\end{equation}
a result recently obtained by Manako[2, p.507, eqn.(11)].\\
In particular in \eqref{mainsc2a}, if we take $c= 2a$,  we get 
\begin{equation} \label{mainsc2b}
\phi_2(a, a; 2a; x, -x)  = \;  _{0}F_1\left[\begin{array}{c} - \\ a+\frac{1}{2} \end{array} ; \frac{x^2}{4}\right] 
\end{equation}
%=====================
Similarly other results can be obtained.
%========================================
\section*{Concluding Remark}
%======================================
In this research note, we have obtained the result\\\\  
\textbf{(A)} $ \; \phi_2(a, b; c; x,y)  = \sum_{m=0}^\infty \frac{(a)_m}{(c)_m} \; _{2}F_1\left[\begin{array}{c}-m, b\\ 1-a-m \end{array} ; \frac{y}{x}\right] \frac{x^m}{m!}$\;, 
for $c \neq 0, -1, -2, \cdots$\\\\
A similar result for $\Psi_2$ has recently been  obtained by Manako[2, p.506, eqn.(6)] as \\\\
 %\begin{equation*} 
\textbf{(B)} $ \; \Psi_2(a; b, c; x,y)  = \sum_{n=0}^\infty \frac{(a)_n}{(b)_n} \; _{2}F_1\left[\begin{array}{c}-n, -n-b+1\\ c \end{array} ; \frac{y}{x}\right] \frac{x^n}{n!}$\\\\
%\end{equation*}
We conclude this research note  by remarking that by employing the results (A) and (B), explicit expressions of Humbert's functions $\phi_2$ and $\Psi_2$ in the general form
\begin{enumerate}[(i)]
 \item $ \phi_2(a, a+i ; c ; x , -x)$ \\  
 and 
\item  $ \Psi(a; c , c+i;    x , -x)$ 
\end{enumerate}
each for $ i = 0, \pm 1, \ldots, \pm 5$\\ 
are under investigations and  will be published soon.
%%%%%%%%%%%%%%%%%%%%%%%%%%%%%%%%%%%%%%%%%%%%%%%%%%%%%%%%%%%%%%%%%%%%%%%%%%%%%%%%%%%%%%%%%%%%%%%%%%%%%%%%%%%%%%%%%%

\end{document}